\documentclass[a4]{amsart}

\title{On closed  Lie ideals of certain tensor products of
  $C^*$-algebras}

\author[V P Gupta]{Ved Prakash Gupta} \address{School of Physical
  Sciences,\\ Jawaharlal Nehru University\\ New Delhi-110067, INDIA.}
\email{vedgupta@mail.jnu.ac.in}

\author[R Jain]{Ranjana Jain} \address{Department of
  Mathematics\\ University of Delhi\\ Delhi-110007, INDIA.}
\email{rjain@maths.du.ac.in} \thanks{The first named author was
  supported partially by a UPE II project (with Id 228) of Jawaharlal
  Nehru University, New Delhi and the second named author was supported by individual R $\&$ D
  grants 2014-15 and 2015-16 of University of Delhi, Delhi.}

\usepackage[all]{xy}
\usepackage{amsmath,amsfonts,amssymb,amsthm,a4,hyperref}
\usepackage[margin=3.9cm]{geometry}
\usepackage{hyperref}
\usepackage{cleveref}
\usepackage{setspace}
\usepackage{enumitem}

\newtheorem{thm}{\sc Theorem}[section]
\newtheorem{cor}[thm]{\sc Corollary}
\newtheorem{prop}[thm]{\sc Proposition}
\newtheorem{lemma}[thm]{\sc Lemma}
\theoremstyle{remark}
\newtheorem{defn}[thm]{\sc Definition}
\newtheorem{rem}[thm]{\sc Remark}

\numberwithin{equation}{section}

\newcommand{\oop}{\widehat\otimes}
\newcommand{\seq}{\subseteq}
\newcommand{\oh}{\otimes^h}
\newcommand{\omin}{\otimes^{\min}}
\newcommand{\omax}{\otimes^{\max}}
\newcommand{\obp}{\otimes^\gamma}
\newcommand{\oi}{\otimes^\lambda}
\newcommand{\ra}{\rightarrow}
\newcommand{\ot}{\otimes}

\newcommand{\C}{\mathbb{C}}

\newcommand{\ol}{\overline}
\newcommand{\mcal}{\mathcal}

\begin{document}



\begin{abstract}
For a simple $C^*$-algebra $A$ and any other $C^*$-algebra $B$, it is
proved that every closed ideal of $A \omin B$ is a product ideal if
either $A$ is exact or $B$ is nuclear. Closed commutator of a closed
ideal in a Banach algebra whose every closed ideal possesses a
quasi-central approximate identity is described in terms of the
commutator of the Banach algebra.  If $\alpha$ is either the Haagerup
norm, the operator space projective norm or the $C^*$-minimal norm,
then this allows us to identify all closed Lie ideals of $A \ot^\alpha
B$, where $A$ and $B$ are simple, unital $C^*$-algebras with one of
them admitting no tracial functionals, and to deduce that every
non-central closed Lie ideal of $B(H) \ot^\alpha B(H)$ contains the
product ideal $K(H) \ot^\alpha K(H)$.  Closed Lie ideals of $A \omin
C(X)$ are also determined, $A$ being any simple unital $C^*$-algebra
with at most one tracial state and $X$ any compact Hausdorff
space. And, it is shown that closed Lie ideals of $A \ot^\alpha K(H)$
are precisely the product ideals, where $A$ is any unital
$C^*$-algebra and $\alpha$ any completely positive uniform tensor
norm.
\end{abstract}

\maketitle

\section{\bf {Introduction}}
A complex associative algebra $A$ inherits a canonical Lie algebra
structure given by the bracket $A \times A \ni (x, y) \mapsto [x,y]:=
xy-yx \in A$ and a subspace $L$ of $A$ is said to be a {\em Lie ideal}
if $[a, x] \in L$ for all $a \in A$ and $x \in L$.

Analysis of ideal structures of various tensor products of operator
algebras has been an important project and a good deal of work has
been done in this direction - see, for instance, \cite{ass, guichardet, 
  jk11, JK-edin, Kum01, kr-1, Tak}.  On the other hand, there
also exists an extensive literature devoted towards the study of Lie
ideals, directly as well as through ideals of the algebra, in pure as
well as Banach and operator algebras - see \cite{bresar08, fong,
  Mar95, Mar10, MM, Miers} and the references therein.

The analysis of closed Lie ideals in operator algebras is primarily
motivated by the evident relationship between commutators, projections
and closed Lie ideals in $C^*$-algebras. For instance, Pedersen
(\cite[Lemma 1]{Ped80}) showed that the closed subspace
$L(\mathcal{P})$ and the $C^*$-subalgebra $A(\mathcal{P})$ generated
by the set of projections $\mcal{P}$ of a $C^*$-algebra $A$ are both
closed Lie ideals of $A$; and, moreover, if $A$ is simple with a
non-trivial projection and if $A$ has at most one tracial state then
$A = L(\mathcal{P})$, i.e., the span of the projections is dense in
$A$ (\cite[Corollary 4]{Ped80}).

However, unlike the ideals of tensor products of operator algebras,
not much is known about the closed Lie ideals of various tensor
products of operator algebras. Among the few known results in this
direction, Marcoux \cite{Mar95}, in 1995, proved that for a UHF
$C^*$-algebra $A$, a subspace $L$ of $A \omin C(X)$ is a closed Lie ideal
if and only if
 $$ L = \ol{sl(A) \ot J} + \C I \otimes S$$ for some closed ideal $J$
and some closed subspace $S$ in $C(X)$, where $sl(A) := \{a \in A:
tr_A(a)=0\}$ with respect to the unique faithful tracial state $tr_A$
on $A$. Then, in 2008, relying heavily on the Lie ideal structure of
tensor products of pure algebras, Bre\v{s}ar et al., in
\cite{bresar08}, proved that for a unital Banach algebra $A$, the
closed Lie ideals of $A \omin K(H)$, of the Banach space projective
tensor product $A \obp K(H)$ and of the Banach space injective tensor
product $A \oi K(H)$ (if it is a Banach algebra) are precisely the
closed ideals.

In this article, we focus on analyzing the (closed) ideal and Lie ideal
structures of certain tensor products of $C^*$-algebras. Here is a
quick overview of the structure of this paper.

In \Cref{simple}, we generalize a characterization (of \cite{fong,
  Mar95}) for closed Lie ideals via invariance under unitaries in a
simple unital $C^*$-algebra containing non-trivial projections and
admitting at most one tracial state. Then, in \Cref{ideal-structure},
following the footsteps of \cite{ass, JK-edin}, for a simple
$C^*$-algebra $A$ and any $C^*$-algebra $B$ we discuss the ideal
structure of $A \omin B$ when $A$ is exact or $B$ is nuclear.

\Cref{q-c-a-i} is the key part of this article. Starting with the
analysis of closed commutators of closed ideals, we move on to obtain
a generalization (see \Cref{ideal-lie-ideal}) of a characterization of
closed Lie ideals in $C^*$-algebras given by Bre\v{s}ar et
al.~\cite{bresar08} to Banach algebras in which sufficiently many
closed ideals possess {\em quasi-central approximate
  identities}. Using these, when $\alpha$ is either the Haagerup norm,
the operator space projective norm or the $C^*$-minimal norm, we
identify all closed Lie ideals of $A \ot^\alpha B$, where $A$ and $B$
are simple, unital $C^*$-algebras with one of them admitting no
tracial functionals, and, deduce that $B(H) \ot^\alpha B(H)$ has only
one non-zero {\em central} Lie ideal, namely, $\C ( 1 \ot 1)$, whereas
every {\em non-central} closed Lie ideal contains the product ideal
$K(H) \ot^\alpha K(H)$.

In \Cref{a-cx}, we basically show that the techniques of Marcoux and
Bre\v{s}ar et al.~ can be applied to obtain an analogy to Marcoux's
result (\cite{Mar95}) that determines the structure of Lie ideals of
$A \omin C(X)$, where $A$ is any simple unital $C^*$-algebra with at
most one tracial state and $X$ is any compact Hausdorff space. And,
finally, in \Cref{cput}, applying a deep result of Bre\v{s}ar et
al.~\cite{bresar08}, we deduce that closed Lie ideals of $A \ot^\alpha
K(H)$ are precisely the product ideals, where $A$ is a unital
$C^*$-algebra and $\alpha$ a {\em completely positive uniform tensor
  norm}.

\section{Closed Lie ideals of simple unital $C^*$-algebras}\label{simple}
In order to maintain distinction between algebraic and topological
simplicity, we shall say that a Banach algebra is {\em topologically
  simple} if it does not contain any non-trivial closed two sided
ideal in it. However, since maximal ideals are closed and every proper
ideal is contained in a maximal ideal in a unital Banach algebra, it is
easily seen that the two notions are same for unital Banach algebras.

Recall that, a {\em tracial state} $\varphi$ on a $C^*$-algebra $A$ is a
positive linear functional of norm one satisfying $\varphi(ab) =
\varphi (ba)$ for all $a, b \in A$. If $A$ is unital, then a tracial
state is unital, i.e., $\varphi (1) = \|\varphi\| = 1$. The collection
of tracial states on $A$ is denoted by $\mcal{T}(A)$.

Note that, for each $\varphi \in \mcal{T}(A)$, $\ker(\varphi)$ is
clearly a closed Lie ideal in $A$ of co-dimension $1$ and contains the
closed {\em commutator Lie ideal} $\overline{[A, A]}$. In particular,
if $\mcal{T}(A) \neq \emptyset$, then $sl(A) := \cap \{ \ker
\,(\varphi) : {\varphi \in \mcal{T}(A)}\}$ is also a closed Lie ideal
in $A$ and contains $\overline{[A, A]}$.  Cuntz and Pederson
(\cite{CP}) proved that they are, in fact, equal. 

\begin{thm}\label{cp} (\cite[Theorem 2.9]{CP}, \cite[Theorem 1]{Pop})
  Let $A$ be a  $C^*$-algebra. Then, the following hold:
  \begin{enumerate} \item  $ \overline{[A,
     A]}$ = $\left\{ \begin{array}{ll}
      sl(A) & \mathit{if}\     \mcal{T}(A) \neq
 \emptyset,\ \mathit{and} \\ A & \mathit{if}\     \mcal{T}(A) =
 \emptyset.
    \end{array}
    \right.$
\item  If $A$ is unital and $\mcal{T}(A) =
  \emptyset$, then ${[A,A]} = A$.
\end{enumerate}
  \end{thm}

It turns out that $\overline{[A, A]}$ is the only non-trivial closed
Lie ideal for a large class of $C^*$-algebras. The following
identifications of Lie ideals were made in \cite[Theorem 2.5]{MM} and
\cite[Proposition 5.23]{bresar08}, and we will require this list in our
discussions ahead.

\begin{prop}(\cite{MM, bresar08})\label{ideal-factor}
  Let $A$ be a simple unital $C^*$-algebra.
  \begin{enumerate} 
    \item  If $A$ has no tracial states, then the only Lie
      ideals of $A$ are $\{0\}$, $\mathbb{C} 1$ and $A$.
    \item If $A$ has a unique tracial state, then the only closed Lie
      ideals of $A$ are $\{0\}$, $\mathbb{C} 1$, $ \text{sl}(A)$ and $A$.
      \end{enumerate}
  \end{prop}

 \begin{cor}\label{cor-simple-lie}
If $M$ is a $II_1$-factor or an $I_n$-factor, then the only (uniformly)
closed Lie ideals of $M$ are $\{0\}$, $\mathbb{C} 1$, $\text{sl}(M)$
and $M$.
\end{cor}
 
\begin{proof}
 A $II_1$-factor or an $I_n$-factor is algebraically simple because it
 is a simple unital $C^*$-algebra - see
 \cite[III.1.7.11]{bd}. Moreover, it has a unique tracial state - see \cite[III.2.5.7]{bd}.
\end{proof}

 Fong, Miers and Sourour (\cite[Theorem 1]{fong}) and Marcoux
 (\cite[Theorem 2.12]{Mar95}) characterized closed Lie ideals of
 $B(H)$ and of a UHF $C^*$-algebra, respectively, through invariance
 under unitary conjugation. Note that $B(H)$ admits no tracial states
 and a UHF $C^*$-algebra admits a unique tracial state and both are
 spanned by their projections (\cite{Mar10}, \cite[Theorem 4.6]{MM}).
 Imitating the original proofs, we obtain the following generalization
 of above characterization.

\begin{prop}\label{invariance}
  Let $A$ be a simple unital $C^*$-algebra with at most one
  tracial state and suppose it contains a non-trivial
 projection. Let $L$ be a closed subspace of $A$. Then the following
 are equivalent:
 \begin{enumerate} 
  \item $L$ is a Lie ideal.
  \item $s^{-1} Ls \subseteq L$ for all invertible elements $s$ in $A$.
  \item $u^*Lu\subseteq L$ for all unitaries $u$ in $A$.
 \end{enumerate}  
\end{prop}
\begin{proof}
By \Cref{ideal-factor}, $(1) \Rightarrow (2)$ is a straight forward
verification on the possible list of closed Lie ideals. The
implication $ (2)\Rightarrow (3)$ is obvious.

  In order to show $(3) \Rightarrow (1)$, note that for every
  projection $p \in A$, $u := p+ i(1-p)$ is a unitary and for $l \in
  L$,
  $$ [p,l] = \frac{(u^*lu-ulu^*)}{2i} \in L.  $$ Then, since $A$ is a
  simple unital $C^*$-algebra with either no tracial states or a
  unique tracial state and contains a non-trivial projection, the
  projections span a dense subspace of $A$ (\cite[Corollary 7]{Ped80}),
  and we are done.
\end{proof}

We now show that the analogue of \Cref{invariance} does not hold in
Banach algebras. Recall (from \cite{bresar08}) that {\em a tracial
  functional} on a Banach algebra $A$ is a non-zero continuous linear
functional $\varphi$ satisfying $\varphi (a b) =\varphi (ba)$ for all
$a, b \in A$. The collection of tracial functionals on $A$ is denoted
by $\mcal{TF}(A)$. By Hahn-Banach Theorem, we easily see that
$\overline{[A, A]} = A$ if and only if $\mcal{TF}(A) = \emptyset$.

In a unital Banach algebra $A$, the set of its unitaries is defined as
$U(A) = \{ u \in GL(A): \| u \| = 1 = \| u^{-1}\| \}$.  If $A$ is a
unital $C^*$-algebra, then clearly $\{u \in A: u u^* = 1 = u^* u \}
\subseteq U(A)$ and for $u \in U(A)$, considering $A \subseteq B(H)$
for some Hilbert space $H$, we see that $\|\xi \| = \| u^{-1} u
(\xi)\| \leq \| u(\xi)\| \leq \|\xi\|$ for all $\xi \in H$, so that
$u$ is an isometry. In particular, it follows that for a unital
$C^*$-algebra $A$, both  definitions give the same set, i.e., $U(A) = \{u
\in A: u u^* = 1 = u^* u \} .$

\begin{rem}\label{invariance-failure}
For any two unital $C^*$-algebras $A$ and $B$, it is known
(\cite[Corollary 2]{jk08}) that $U(A \oh B) = \{ u \ot v: u \in U(A),
v \in U(B)\}$, where $\oh$ is the Haagerup tensor product (see
\cite{ERbook}). By a result of Fack (see \cite[Theorem 2.16]{Mar10}),
the Cuntz algebra $\mcal{O}_2$ is spanned by its commutators and,
therefore, it has no tracial functionals. Further, since $\|\cdot\|_h$
is cross norm (see \cite{ERbook}), $\mcal{O}_2 \ni x \ra x \ot 1 \in
\mcal{O}_2 \oh \mcal{O}_2$ is an isometric homomorphism, so the Banach
algebra $\mcal{O}_2 \oh \mcal{O}_2$ does not have any tracial
functionals, as well. Also, since $\mcal{O}_2$ is a simple
$C^*$-algebra, $\mcal{O}_2 \oh \mcal{O}_2$ is a topologically simple
Banach algebra, by \cite[Theorem 5.1]{ass}.  By above decomposition of
unitaries, $\mcal{O}_2 \ot \C1$ is invariant under conjugation by
unitaries in $\mcal{O}_2 \oh \mcal{O}_2$ but it is easily seen that it
is not a Lie-ideal.
\end{rem}

 On similar lines, for any infinite dimensional Hilbert space $H$, it
 can also be seen that $K(H) \ot \C 1$ is invariant under conjugation
 by unitaries in $B(H) \oh B(H)$ but is not a Lie ideal. These
 observations also illustrate that tensor product of two Lie ideals need not be a
 Lie ideal.

\begin{rem}
Unlike the above decomposition of unitaries in the Banach algebra
$\mcal{O}_2 \oh \mcal{O}_2$, the unitaries in the $C^*$-algebra
$\mcal{O}_2 \omin \mcal{O}_2$ do not decompose as elementary
tensors. Indeed, since $\mcal{O}_2 \omin \mcal{O}_2$ is a simple (see
\cite{Tak}), unital $C^*$-algebra and has no tracial states, by
\cite[Proposition 5.23]{bresar08} or \Cref{A-B-simple} below, its only
closed Lie ideals are $\{0\}$, $\C (1 \ot 1)$ and itself. Since
$\mcal{O}_2$ contains non-trivial projections, so does $\mcal{O}_2
\omin \mcal{O}_2$; therefore, by \Cref{invariance}, $\mcal{O}_2 \ot \C
1$ is not invariant under conjugation by unitaries. In particular, not
every unitary in $\mcal{O}_2 \omin \mcal{O}_2 $ can be expressed as an
elementary tensor $u \ot v$ for unitaries $u$ and $v$ in $\mcal{O}_2$.
\end{rem}

\section{Closed Ideals of $A \omin B$}\label{ideal-structure}

Let $A$ and $B$ be $C^*$-algebras and suppose $A$ is topologically
simple. If $\alpha$ is either the Haagerup tensor product or the
operator space projective tensor product, then by \cite[Proposition
  5.2]{ass}, and by \cite[Theorem 3.8]{JK-edin}, it is known that
every closed ideal of the Banach algebra $A \ot^\alpha B$ is a product
ideal of the form $ A \ot^\alpha J$ for some closed ideal $J$ in $B$.

In general, not much is known about the ideal structure of the
$C^*$-minimal tensor product. However, the (Zorn's Lemma) technique
used in above ideal structures can be applied to analyze the ideals of
$A\omin B$ under some additional hypothesis, which we demonstrate
below.

\begin{thm}\label{simple-ideal}
 Let $A$ and $B$ be $C^*$-algebras where $A$ is topologically
 simple. If either $A$ is exact or $B$ is nuclear, then every closed
 ideal of the $C^*$-algebra $A \omin B$ is a product ideal of the form
 $ A \omin J$ for some closed ideal $J$ in $B$.
\end{thm}

\begin{proof}
Let $I$ be a non-zero closed ideal in $A \omin B$.  Consider the
collection $$ \mcal{F} := \{ J \subseteq B: J
\ \mathrm{is\ a\ closed\ ideal\ in}\ B\ \mathrm{and}\ A \omin J
\subseteq I\}.$$ By \cite[Proposition 4.5]{ass}, $I$ contains a
non-zero elementary tensor, say, $a \ot b$. If $K$ and $J$ are the
non-zero closed ideals in $A$ and $B$ generated by $a$ and $b$,
respectively, then by simplicity of $A$, we have $K = A$ and $A \omin
J \subseteq I$. In particular, $\mcal{F} \neq \emptyset$.

Note that, by injectivity of $\omin$ and the fact that a finite sum of
closed ideals is closed in a $C^*$-algebra, it is easily seen that $A
\omin (\sum_i J_i) = \sum_i (A \omin J_i)$ for any finite collection
of closed ideals $\{J_i\}$ in $B$.  So, with respect to the partial
order given by set inclusion, every chain $\{ J_i : i \in \Lambda\}$
in $\mcal{F}$ has an upper bound, namely, the closure of the ideal
$\{\sum_{\mathrm{finite}} x_i : x_i \in J_i\}$ in $\mcal{F}$, implying
thereby that there exists a maximal element, say $J$, in $\mcal{F}$.

We will show that $A \omin J = I$. Consider the map $\mathrm{Id} \omin
\pi : A \omin B \ra A \omin (B / J)$. If $A$ is exact, then by
definition of exactness, its kernel is $A \omin J$; and, if $B$ is
nuclear, then so are $J$ and $B/J$ and it is known (see \cite{bd, guichardet})
that the sequence
$$0 \ra A \omax J \ra A \omax B \ra A \omax (B/J) \ra 0$$ is always
exact and, therefore, we obtain $$\ker(\mathrm{Id}\omin \pi) =
\ker(\mathrm{Id}\omax \pi) = A \omax J = A \omin J.$$ Since
$\mathrm{Id} \omin \pi$ is a surjective $*$-homomorphism,
$\widetilde{I}:=(\mathrm{Id} \omin \pi)(I)$ is a closed ideal in $A
\omin (B / J)$. It is now sufficient to show that this is the zero
ideal. If $\widetilde{I} \neq 0$, then, again by \cite[Proposition
  4.5]{ass}, $\widetilde{I}$ contains a non-zero elementary tensor,
say, $a \ot (b +J)$. Let $K$ be the closed ideal in $B$ generated by
$b$. Since $A$ is simple, it equals the closed ideal generated by $a$
and we obtain $A \omin K \subseteq I$, a contradiction to the maximality
of $J$ as $A \omin K$ is not contained in $A \omin J$.
\end{proof}

 \section{Ideals with quasi-central approximate identities and their closed commutators}\label{q-c-a-i}
 
We first recall some definitions and notations from \cite{Mar95,
  bresar08}.  Every subspace of $Z(A)$, the center of an associative
algebra $A$, is clearly a Lie ideal in $A$ and is called a {\em
  central Lie ideal}.  For subspaces $X$ and $Y$ of $A$, $$[X,Y]:=
\text{span} \{ [x,y]: x\in X, y \in Y\}\ \text{ and } XY:= \text{span}
\{ xy: x \in X, y \in Y\}.$$ If $L$ and $M$ are Lie ideals in $A$ then
so is $[L, M]$. For a subspace $S$ of  $A$, consider the
subspace $$N(S) : = \{x \in A: [x, a] \in S \ \mathrm{for\ all}\ a \in
A \}.$$ If $L$ is a Lie ideal then $N(L)$ is a subalgebra as well as a
Lie ideal of $A$ (\cite[Proposition 2.2]{bresar08}). Note that if $I$
is an ideal in $A$, then any subspace $L$ of $A$ {\em embraced by
  $I$}, i.e., satisfying $[I, A] \subseteq L \subseteq N([I, A])$, is
a Lie ideal in $A$. In fact, Bre\v{s}ar et al.~ \cite[$\S
  5$]{bresar08} showed that a closed subspace $L$ of a $C^*$-algebra
$A$ is a Lie ideal if and only if it is {\em topologically embraced}
by a closed ideal $I$ in $A$, i.e.,
 $$\ol{[I, A]} \subseteq L \subseteq N(\ol{[I, A]}).$$ We show below
(see \Cref{ideal-lie-ideal}) that this characterization generalizes to
Banach algebras whose every closed ideal possesses a {\em
  quasi-central approximate identity}. Examples of such Banach
algebras (which are not $C^*$-algebras) will be illustrated in
\Cref{comm-section}.

For a closed ideal $I$ in a $C^*$-algebra $A$, it
is known (\cite[Lemma 1]{Miers} and \cite[Proposition 5.25]{bresar08})
that
$$\overline{[I, I]}= \ol{[[I, A], A]} = \overline{[I, A]} = I\, \cap\,
\overline {[A, A]}.$$ Miers (in \cite{Miers}) mentions that the third
equality was due to Bunce and gives a proof using {\em quasi-central
  approximate identity}, and the other two equalities were proved by
Bre\v{s}ar et al.~using techniques of von Neumann algebras.  We
generalize this result to ideals in Banach algebras with {
  quasi-central approximate identities. The proof given here borrows
  ideas from \cite{Miers, Rob} and does not involve von Neumann
  algebra tools. 
  \vspace*{1mm}

\begin{defn}\cite{ass}  If $I$ is an ideal in a
Banach algebra $A$, then a net $\{e _\lambda\}$ in $I$ is said to be a
{quasi-central approximate identity for $I$ in $A$} if
 \begin{enumerate} 
\item $\sup_\lambda \| e _\lambda\| < \infty$, and 
\item  $ \lim_\lambda \| x e_\lambda
  - x \| = \lim_\lambda \| e_\lambda x - x \| = \lim_\lambda
  \|e_\lambda a - a e_\lambda \| = 0$  for all $x \in I$ and $a \in A$.
 \end{enumerate}
\end{defn}

 It is known that all ideals (not necessarily closed) in
 $C^*$-algebras possess quasi-central approximate identities
 (\cite[Theorem 3.2]{AP} and \cite[Theorem 1]{arveson}).

 The following equalities between commutators of ideals will be
 required ahead in a characterization of closed Lie ideals (see
 \cite[Lemma 1]{Miers}, \cite[Proposition 5.25]{bresar08} and
 \cite[Lemma 1.4]{Rob} for ideals in $C^*$-algebras).

 \begin{lemma}\label{commutator}
Let $I$ be a closed ideal in a Banach algebra $A$. If $I$ admits a
quasi-central approximate identity in $A$,
then \begin{equation}\label{comm} \ol{[I, I]} = \ol{[I, A]} = I\,
  \cap\, \overline {[A, A]}.\end{equation} In particular, $N(\ol{[I,
    A]}) = N(I)$ and, if $A$ has no tracial functionals, then $\ol{[I,
    A]} = I$.  Moreover, if the closed ideal $J := \ol{\mathrm{Id}[I,A]}$
also contains a quasi-central approximate identity, then $\ol{[I, A]} =
\ol{[J, A]}$.
\end{lemma}

\begin{proof}
Let $\{e_\lambda\}$ be a quasi-central approximate identity for the
ideal $I$ in $A$.  Since $I$ is a closed ideal, clearly $$\ol{[I, I]}
 \subseteq \overline{[I, A]} \subseteq I\,
  \cap\, \overline {[A, A]}.$$ For the reverse inclusions, we first
  show that $I \cap \ol{[A, A]} \subseteq \ol{[I, A]}$.  Let $z \in I
  \cap \ol{[A, A]}$ and $\epsilon > 0$. Then, there exist $x_i, y_i
  \in A$, $ 1 \leq i \leq n$ such that $\| z - \sum_i [x_i, y_i]\| <
  \epsilon/3$.  Note that $\sum_i [x_i e_\lambda, y_i] \in [I, A]$ for all $\lambda$, and
 $$\begin{array}{ccl} \| z - \sum_i [x_i e_\lambda, y_i]\|  & \leq & \| z
   - z e_\lambda \|  + \| z e_\lambda + \sum_i x_i [e_\lambda, y_i ] - \sum_i
   [x_i e_\lambda, y_i]\|  + \| \sum_i x_i[e_\lambda, y_i]\|  \\ & = & \| z
   - z e_\lambda \|  + \| z e_\lambda - \sum_i[x_i, y_i] e_\lambda \|  +
   \sum_i\|x_i\|  \| e_\lambda y_i - y_i e_\lambda \|. 
  \end{array} $$
Thus, there exists an index $\lambda_0$ such that $\| z - \sum_i [x_i e_{\lambda_0},
  y_i]\|  < \epsilon$ implying that $ z \in \ol{[I,A]}$.

For the remaining equality, it suffices to show that $[I, I]$ is
dense in $\ol{[I, A]}$. Let $w \in \ol{[I, A]}$ and $\epsilon >
0$. Then there exist $u_i \in I, a_i \in A $, $ 1 \leq i \leq n$ such
that $\| w - \sum_i [u_i, a_i]\| < \epsilon/3$. Clearly, $\sum_i [u_i
  , e_\lambda a_i] \in [I, I]$ and, as above, it is easily seen that
$$ \| w - \sum_i [u_i , e_\lambda a_i]\| \leq \| w - e_\lambda w\| +
\| e_\lambda w - e_\lambda \sum_i [u_i, a_i] \| + \sum_i\|u_i\| \|
e_\lambda a_i - a_i e_\lambda \|,
    $$
    implying that $[I, I]$ is dense in $\ol{[I, A]}$.

    Since $\ol{[I, A]} \subseteq I$, by definition, $N(\ol{[I,
        A]}) \subseteq N(I)$ and if $x \in N(I)$, then $[x, A] \subseteq I
    \cap [A, A] \subseteq \ol{[I, A]}$ implying that $x \in N(\ol{[I,
        A]})$ and hence $N(\ol{[I, A]}) = N(I)$.
\vspace*{2mm}

If $A$ has no tracial functionals, then $\ol{[A,A]} = A$ and,
therefore,  $\ol{[I, A]} = I \cap \ol{[A, A]} =
I$.
\vspace*{2mm}

 Finally, suppose  the closed ideal $J :=\ol{\mathrm{Id}[I,
      A]}$ admits a quasi-central approximate
  identity, say,   $\{f_\mu\}$.    Since $J \subseteq I$, clearly $\ol{[J, A]} \subseteq
  \ol{[I, A]}$.  Let $x \in I$ and $a \in A$. Then, $[x, a ] \in J$,
  $[f_\mu x, a]\in [J, A]$ and \
\begin{eqnarray*}
  \| [f_\mu x , a] - [x , a]\| & \leq & \| [f_\mu x , a] -
  f_\mu [x , a ]\| + \| f_\mu[x, a] - [x , a]\| \\ & = & \|
  f_\mu a - a f_\mu \|\, \|x\| + \| f_\mu[x, a] - [x ,
    a]\| \ra 0,
  \end{eqnarray*}
implying that $[J, A]$ is dense in $[I, A]$ and hence $\ol{[I, A]} =
\ol{[J, A]}$. 
\end{proof}

 More generally, using a result by Robert \cite{Rob}, we shall show below that
 $\overline{[L,A ]} = \overline{[\text{Id}[L,A], A]}$ for any closed Lie
 ideal $L$ in an appropriate Banach algebra $A$, which generalizes
 \cite[Theorem 5.27]{bresar08}. Robert, in \cite{Rob}, has given a
 simpler proof of \cite[Theorem 5.27]{bresar08} avoiding von
 Neumann algebra tools.

Recall that a Banach algebra $A$ is said to be semiprime if $I^2 =
(0)$ implies $I = (0)$ for any closed ideal $I$. And a closed ideal
$I$ in $A$ is said to be semiprime if the quotient Banach algebra
$A/I$ is semiprime. A $C^*$-algebra and all its closed ideals are
easily seen to be semiprime.

\begin{lemma}\label{semiprime}
Let $A$ be a Banach algebra whose every closed ideal possesses a left or a
right approximate identity. Then $A$ is semiprime and so are its
closed ideals.
  \end{lemma}
\begin{proof}
  Let $I$ be a closed ideal in $A$ such that $I^2 = (0)$. Let $x \in I$ and
  $\{e_\lambda\}$ be a right approximate identity in
  ${I}$. Then, $x e_\lambda = 0 $ for all $\lambda$ and as $x
  e_\lambda \ra x$, we get $x = 0$ implying that $I =0$. Thus, $A$ is
  semiprime.

  Next, for a closed ideal $I$ in $A$, every closed ideal in $A/I$ is
  of the form $J/I$ for some closed ideal $J$ in $A$ containing
  $I$. If ${J}$ admits a left or a right approximate identity, so does
  ${J/I}$. Therefore, every closed ideal in $A/I$ admits a left or a
  right approximate identity and, as above, $A/I$ is semiprime.
  \end{proof}

We will need the following observation 
by Br\v{e}sar et al. \cite[Proposition 5.2]{bresar08}.
\begin{prop}\label{b-semiprime}
Let $L$ be a closed Lie ideal in a Banach algebra $A$ and $I_L$ 
denote the closed ideal generated by ${[L, L]}$, i.e.,
$I_L:=\ol{\mathrm{Id}[L, L]}$. If the Banach algebra $A/I_L$ is
semiprime or commutative, then $[L, A] \subseteq I_L$.
  \end{prop}

The following mildly generalizes \cite[Theorem 5.27]{bresar08} and a
part of \cite[Theorem 1.5]{Rob}.

\begin{thm}\label{bresar2}
   Let $L$ be a closed Lie ideal in a Banach algebra $A$ and $I$
    denote the closed ideal generated by ${[L, A]}$, i.e.,
    $I:=\ol{\mathrm{Id}[L, A]}$.  If all closed ideals of $A$
    containing $I_L$ possess quasi-central approximate identities, then
    $I = I_L$ and
\begin{equation} \ol{[I, A]} = \ol{[L, A]}.\end{equation}
 \noindent In particular, $\ol{[I, A]} \subseteq L \subseteq N(\ol{[I,
     A]})$. Moreover, if $A$ has no tracial functionals, then $I
 \subseteq L \subseteq N(I)$, as well.
\end{thm}

\begin{proof}
By \Cref{semiprime}, $A/I_L$ is semiprime.  So, by \Cref{b-semiprime}, $[L, A] \subseteq
I_L$ implying that $I = I_L$.  It is
elementary to see that $\text{Id}[L, L] \subseteq {L + L^2}$ (see
\cite[Lemma 1.4]{Rob}), so that $ I\subseteq \ol{L + L^2}$ and,
therefore, $\ol{[I, A]} \subseteq \overline{[L + L^2 , A]} = \ol{[L,
    A]}$, where the last equality follows from the easily verifiable
fact that $[L^2, A] \subseteq [L, A]$ (see \cite[(1.1)]{Rob}).

On the other hand, since $I_L = I$, $I$ contains a quasi-central
approximate identity and, $\ol{[L, A]}
\subseteq I \cap \ol{[A, A]} = \ol{[I, A]}$, by \Cref{commutator}.

The remaining then follows again from \Cref{commutator}.
  \end{proof}

\begin{cor}\label{non-central}
  Let $A$ be a Banach algebra whose every closed ideal contains a
  quasi-central approximate identity and suppose $\mcal{TF}(A) =
  \emptyset$. Then every non-central closed Lie ideal of $A$, i.e., $L
  \nsubseteq Z(A)$, contains a non-zero closed ideal.
\end{cor}

\Cref{bresar2} partially answers a question of Bre\v{s}ar et
al.~\cite[page 120]{bresar08} where they ask for suitable conditions
in Banach $*$-algebra setting so that a closed Lie ideal is closed
commutator equal to a closed ideal, and it also yields the following
characterization of closed Lie ideals:

\begin{cor}\label{ideal-lie-ideal}
If every closed ideal in a Banach algebra $A$ admits a quasi-central
approximate identity, then a closed subspace $L$ of $A$ is a Lie ideal
if and only if there exists a closed ideal $I$ in $A$ such that
$$ \ol{[I, A]} \subseteq L \subseteq N(\ol{[I, A]}). $$
\end{cor}

The techniqe of Robert \cite{Rob}, based on a Theorem of Herstein
\cite{Her}, yields a  stronger version of \Cref{bresar2}.
\begin{thm}\label{robert}
  Let $L$ be a closed Lie ideal in a Banach algebra $A$, $I$ denote
  the closed ideal generated by ${[L, A]}$ and $M$ denote the closed
  Lie ideal $\ol{[L,A]}$. If all closed ideals of $A$ containing $I_M
  := \overline{\mathrm{Id}[M, M]}$ possess quasi-central approximate
  identities, then
  \begin{equation}\label{id-l-a}
    \qquad \qquad I = \ol{[L, A] + [L, A]^2} = B([L,A])\ \mathrm{and}
  \end{equation}
  \begin{equation}
    \ol{[I, A]} = \ol{[L, A]} = \ol{[[L, A], A]},
  \end{equation}
  where $B([L,A])$ denotes the Banach subalgebra of $A$ generated by
  $[L,A]$.
\end{thm}

\begin{proof}
By Lemma \ref{semiprime}, the quotient $A/I_M$ is semiprime. So, the
proof of the equalities $I = \ol{[L, A] + [L, A]^2} = B([L,A])$ given
by Robert in \cite[Theorem 1.5 (i)]{Rob} works verbatim.

Then, the inclusion $\ol{[I, A]} \subseteq \ol{[[L, A], A]} \subseteq
\ol{[L, A]}$ is immediate. Since $I_M \subseteq I$, $I$ contains a
quasi-central approximate identity, so by \Cref{commutator}, we have
$\ol{[I, A]} = I \cap \ol{[A, A]}$ and, since $[L,A] \subseteq I \cap
[A,A]$, the reverse inclusion follows.
\end{proof}

By \Cref{commutator} and \Cref{cp}, \Cref{robert} immediately yields the following:
  
\begin{cor}\label{commutator1}
Let $A$ be a Banach algebra whose every closed ideal admits a
quasi-central approximate identity.  Then, for any closed ideal $I$ in
$A$, we have
 \begin{equation}\label{comm1}\ol{[I, I]} = \ol{[I, A]} = \ol{[I, [I,A]]} =  I\,
  \cap\, \overline {[A, A]}.\end{equation} In particular,  if
 $A$ is a $C^*$-algebra with no tracial states, then $\ol{[I, A]} =
 I$.
\end{cor}

This yields the following generalization of \cite[Corollary
  5.26]{bresar08} and, using \Cref{commutator} and \Cref{commutator1},
the same proof works verbatim.

\begin{cor}
Let $I$ (resp., $L$) be a closed ideal (resp., Lie ideal) in a Banach
algebra $A$. If every closed ideal of $A$ possesses a quasi-central
approximate identity, then { the following are equivalent:}
\begin{enumerate}
\item $[I, A] \subseteq L \subseteq N(I)$.
  \item $\ol{[I, A]} \subseteq L \subseteq N(\ol{[I,A]})$.
\item $ \ol{[I, A]} = \ol{[L, A]}.$
\end{enumerate}
\end{cor}

\subsection{Commutators of closed ideals in certain tensor products of $C^*$-algebras}\label{comm-section}

Apart from the usual spatial tensor product of $C^*$-algebras, we will
also be interested in some tensor products which yield Banach algebras
which are not necessarily $C^*$-algebras. As in \cite[$\S
  2$]{blecher}, a norm $\|\cdot \|_{\alpha}$ on the algebraic tensor
product $A \ot B$ of a pair of $C^*$-algebras $A$ and $B$ is said to
be
\begin{enumerate}
\item a {\em sub-cross norm} if $\|a \ot b\|_\alpha \leq \|a\|\,
  \|b\|$ for all $a \in A$, $b \in B$,
\item an {\em algebra norm} if $\|w\, z \|_{\alpha}  \leq \| w \|_{\alpha}
  \, \|z\|_{\alpha} $ for all $w, z \in A \ot B$, and 
\item a {\em tensor norm} if $\| \cdot \|_{\lambda} \leq \| \cdot
  \|_{\alpha} \leq \| \cdot \|_{\gamma}$, where $\lambda$ and $\gamma$
  are the Banach space injective and projective norms, respectively.
\end{enumerate}
Clearly, $A \otimes^{\alpha} B$, the completion of  $A \ot B$ with
respect to any algebra norm $\|\cdot\|\alpha$, is a Banach algebra. Since $\| \cdot
 \|_\gamma$ is a cross norm, every tensor norm is, therefore, sub-cross.

 The tensor products that we will be concerned with here include the
 $C^*$-minimal tensor product ($\omin$), the (operator space) Haagerup
 tensor product ($\oh$), the operator space projective tensor product
 ($\oop$) { and the Banach space projective tensor product
   ($\ot^\gamma$)}.  We refer the reader to \cite{ERbook, guichardet}
 for their definitions and essential properties. All these norms are
 sub-cross algebra tensor norms and yield Banach algebras. In fact,
 for any pair of $C^*$-algebras,  $\ot^\gamma$ (by definition) and
$\oop$  (by \cite{Kum01}) yield Banach $*$-algebras whereas the
 natural involution is not isometric with respect to $\oh$ (\cite{blecher}).

 The following proposition is an immediate generalization of
 \cite[Corollary 3.4]{ass} and yields examples of closed ideals with
 quasi-central approximate identities in Banach algebras which are not
 $C^*$-algebras.
 \begin{prop}\label{quasi-central}
   Let $A$ and $B$ be $C^*$-algebras and $\alpha$ be an algebra tensor
   norm. Let $\{I_i: 1 \leq i \leq n\}$ and $\{J_i: 1 \leq i \leq n\}$
   be closed ideals  in $A$ and $B$, respectively. Then the closed
   ideal $K := \overline{\sum_i I_i \ot J_i}^\alpha$ admits a quasi-central
   approximate identity in $A \ot^\alpha B$.
   \end{prop}
 \begin{proof}
   By \cite[Lemma 3.3]{ass}, the closure of a finite sum of closed
   ideals containing quasi-central approximate identities in a Banach
   algebra also contains a quasi-central approximate identity. And
   since, $\overline{\sum_i I_i \ot J_i}^\alpha =
   \overline{\sum_i\overline{ I_i \ot J_i}^\alpha}^\alpha$ , it is
   enough to show that an arbitrary product ideal $\ol{I \ot
     J}^\alpha$, for ideals $I$ and $J$ in $A$ and $B$, respectively,
   admits a quasi-central approximate identity in $A \ot^\alpha
   B$.

   Let $\{e_\lambda : \lambda \in \Lambda\}$ and $\{f_\gamma : \gamma
   \in \Gamma\}$ be quasi-central approximate identities for $I$ and
   $J$ in $A$ and $B$, respectively (\cite[Theorem 1]{arveson}), with
   $K =\sup_\lambda \|e_\lambda\|$ and $L = \sup_\gamma \|
   f_\gamma\|$. The set $\Lambda \times \Gamma$ inherits a directed
   structure via the partial ordering
   $$ (\lambda_1, \gamma_1) \leq (\lambda_2, \gamma_2)
   \ \mathrm{if\ and \ only\ if}\ \lambda_1 \leq \lambda_2
   \ \mathrm{and} \  \gamma_1 \leq \gamma_2.$$ Let $e_{(\lambda,
     \gamma)}:= e_\lambda$ and $f_{(\lambda, \gamma)}:= f_\gamma$ for
   all $ (\lambda, \gamma) \in \Lambda \times \Gamma $, and set $z_\mu
   = e_\mu \ot f_\mu$ for all $\mu \in \Lambda \times
   \Gamma$. Clearly, $\{e_\mu : \mu \in \Lambda \times \Gamma\}$ and
   $\{f_\mu : \mu \in \Lambda \times \Gamma\}$ are quasi-central
   approximate identities for $I$ and $J$, respectively. We show that
   $\{z_\mu : \mu \in \Lambda \times \Gamma\}$ is a quasi-central
   approximate identiy for $\ol{I \ot J}^\alpha$ in $A \ot^\alpha
   B$. Since $\alpha$ is a sub-cross norm, $\{ z_\mu\}$ is uniformly
   bounded. Let $x = \sum_i u_i \ot v_i \in I \ot J$ and $y =
   \sum_i a_i \ot b_i \in A \ot B$. Then,

   \begin{eqnarray*}
\| x z_\mu - x\|_\alpha & \leq & \sum_i \|(u_i e_\mu -u_i) \ot v_i
f_\mu \|_\alpha + \|u_i \ot (v_i f_\mu - v_i) \|_\alpha \\ & \leq &
\sum_i L \|u_i e_\mu -u_i\|\, \|v_i\| + \|u_i\|\, \|v_i f_\mu - v_i \| \rightarrow 0,
   \end{eqnarray*}
 likewise, $\| z_\mu x - x \|_\alpha \ra 0$,  and
   \begin{eqnarray*}
\|y z_\mu - z_\mu y\| & \leq & \sum_i \| a_i e_\mu - e_\mu a_i\|\,
\|b_i f_\mu\| + \sum_i \| e_\mu a_i\|\, \| b_i f_\mu - f_\mu b_i\| \\
& \leq & \sum_i L \| a_i e_\mu - e_\mu a_i\|\,
\|b_i\| + \sum_i K\|a_i\|\, \| b_i f_\mu - f_\mu b_i\|
\ra
0.
 \end{eqnarray*}
   Since $I\ot J$ (resp., $A \ot B$) is dense in $\ol{I \ot J}^\alpha$
   (resp., $A \ot^\alpha B$), it follows that $\lim_\mu\| x z_\mu -
   x\|_\alpha = \lim_\mu \| z_\mu x - x\|_\alpha = \lim_\mu \|y z_\mu
   - z_\mu y\| =0 $ for all $ x \in \ol{I \ot J}^\alpha$ and $y \in A
   \ot^\alpha B$.
   \end{proof}

 \begin{rem}
Note that in the above theorem, we have actually proved that, if $I$
and $J$ are ideals (not necessarily closed) in $C^*$-algebras $A$ and
$B$, then the (algebraic) product ideal $I \ot J$ admits a quasi-central approximate
identity in $A \ot^\alpha B$.
   \end{rem}

 We can now easily deduce the following:
\begin{cor}\label{comm-tensor}
 Let $A$ and $B$ be $C^*$-algebras and $\alpha$ be an algebra tensor
   norm. Let $\{J_i: 1 \leq i \leq n\}$ and $\{K_i: 1 \leq i \leq n\}$
   be closed ideals  in $A$ and $B$,   respectively. Then, 
   $$ \ol{[I, I]}= \overline{[ I, A\ot^\alpha B]} = I \cap
   \overline{[A\ot^\alpha B, A \ot^\alpha B]} $$ if $I$ is a closed
   ideal in $A \ot^\alpha B$ of any of the following form:
       \begin{enumerate}
       \item $I = \overline{\sum_i J_i \ot K_i}^\alpha$.
   \item $I = \sum_i {J_i \ot^\alpha K_i}$ and $\alpha$ is either the
     Haagerup norm or the operator space projective norm.

 \item $I$ is any closed ideal in $A \ot^\alpha B$, $A$ contains only
   finitely many closed ideals and $\alpha$ is either the Haagerup
   norm or the operator space projective norm.

 \item $I$ is any ideal in $A \ot^\alpha B$ and ${\alpha}$
   is any $C^*$-tensor norm.
       \end{enumerate}
\end{cor}

\begin{proof}\begin{enumerate}
\item is immediate from \Cref{quasi-central} and \Cref{commutator}.
  \vspace*{1mm}
  
\item\hspace*{-2mm}: By \cite[Theorem 3.8]{ass} and \cite[Proposition 3.2]{jk11},
  ${\sum_i J_i \ot^\alpha K_i}$ is a closed ideal in $A \ot^\alpha B$.
  \vspace*{1mm}
  
\item\hspace*{-2mm}: By \cite[Theorem 5.3]{ass} and \cite[Theorem
  3.4]{kr-1}, every closed ideal in $A \ot^\alpha B$ is a finite sum
  of product ideals.
\vspace*{1mm}
  
\item follows from the fact that every ideal in a $C^*$-algebra admits
  a quasi-central approximate identity (\cite{AP, arveson}).
  \end{enumerate}
  \end{proof}

\begin{rem}
If a $C^*$-algebra $A$ ($\ncong \C$) contains only finitely many
closed ideals, then for any $C^*$-algebra $B$ ($\ncong \C$), $A \oh B$
or $A \oop B$ is a Banach algebra which is not a $C^*$-alegbra
(\cite[Theorem 1]{blecher}) and, as seen above, its every closed
ideal possesses a quasi-central approximate identity.
\end{rem}

\begin{prop}\label{non-central-lie-ideal}
  Let $A$ and $B$ be $C^*$-algebras and suppose $B$ is unital. Then,
  every non-central closed Lie ideal in $A \ot^\alpha B$ contains a
  non-zero closed ideal in the following cases:
\begin{enumerate}
\item $A$ has no tracial states and $\alpha $ is any $C^*$-norm.
\item $A$ has no tracial functionals, $A$ contains only finitely many
  closed ideals and $\alpha$ is either the Haagerup norm or the
  operator space projective norm.
  \end{enumerate}
  \end{prop}
\begin{proof}
\begin{enumerate}
\item \hspace*{-2mm}: Since $B$ is unital, for any $C^*$-norm
  $\alpha$, $A \subseteq A \ot^\alpha B$ as a $C^*$-subalgebra, so $A
  \ot^\alpha B$ does not have any tracial states and, therefore, the
  assertion holds by \Cref{bresar2} and \Cref{commutator1}.
\item \hspace*{-2mm}: Since $\alpha$ is a cross norm (see
  \cite{ERbook}), $A \ni a \mapsto a \ot 1 \in A \ot^\alpha B$ is an
  isometric homomorphism; so, the Banach algebra $A \ot^\alpha B$
  admits no tracial functionals. The rest is then taken care of by
  \Cref{comm-tensor}(3) and \Cref{non-central}.
 \end{enumerate}
\end{proof}

\begin{thm}\label{A-B-simple}
 Let $A$ and $B$ be simple, unital $C^*$-algebras and suppose one of
 them admits no tracial functionals. If $\alpha$ is either the
 Haagerup norm, the operator space projective norm or the
 $C^*$-minimal norm, then the only closed Lie ideals of $A \ot^\alpha
 B$ are $\{0\}, \C (1 \ot 1)$ and $A \ot^\alpha B$ itself.
  \end{thm}
\begin{proof}
 If $\ot^\alpha$ is $\oh$ or $\omin$, it is known (\cite[Theorem
   2.13]{ass} and \cite[Corollary 1]{Was}) that $Z(A \ot^\alpha B) =
 Z(A) \ot^\alpha Z(B)$. And, by \cite[Theorem 3]{jk08}, the algebraic
 isomorphism $Z(A) \ot Z(B) \cong Z(A \ot B)$ extends to an algebraic
 isomorphism (not necessarily isometric) between $Z(A) \oop Z(B)$ and
 $Z(A \oop B)$. So, in all three cases, we obtain $Z(A \ot^\alpha B) =
 \C (1 \ot 1)$. In particular, the only central Lie ideals of $A
 \ot^\alpha B$ are $\{0\}$ and $\C (1 \ot 1)$.

The $C^*$-algebra $A \omin B$ is simple (see \cite[Corollary
  IV.4.21]{Tak}).  And, by \cite[Theorem 5.1]{ass} and \cite[Theorem
  3.7]{JK-edin}, the Banach algebras $A \oh B$ and $A \oop B$ are
topologically simple. So, by \Cref{non-central-lie-ideal}, $A
\ot^\alpha B$ is its only non-central closed Lie ideal.
\end{proof}
We conclude this section with the following:

\begin{thm}\label{bh-bh}
Let $H$ be an infinite dimensional separable Hilbert space. If
$\alpha$ is either the Haagerup norm, the operator space projective
norm or the $C^*$-minimal norm, then the only non-zero central Lie
ideal of $B(H) \ot^\alpha B(H)$ is $\C (1 \ot 1)$ and every
non-central closed Lie ideal of $B(H) \ot^{\alpha} B(H)$ contains the
product ideal $K(H) \ot^{\alpha} K(H)$.
  \end{thm}
\begin{proof}
As in \Cref{A-B-simple}, we  obtain $Z(B(H)
 \ot^\alpha B(H)) = \C (1 \ot 1)$.

 Now, let $L$ be a non-central closed Lie ideal in $B(H) \ot^{\alpha}
 B(H)$. By a theorem of Halmos, every bounded operator on $H$ is a sum
 of two commutators, so $B(H)$ does not admit any tracial
 functionals. Thus, $L$ must contain a non-zero closed ideal by
 \Cref{non-central-lie-ideal}. By \cite[Proposition 4.5 and
   Corollary 4.6]{ass} and \cite[Proposition 3.6]{JK-edin}, every
 non-zero closed ideal of $B(H) \ot^\alpha B(H)$ contains an
 elementary tensor, say, $a \ot b$. So, $K(H)$ being the only
 non-trivial closed ideal in $B(H)$, $K(H) \ot^\alpha K(H)$ must be
 contained in $\ol{\mathrm{Id} \{a\}} \ot^\alpha \ol{\mathrm{Id}
   \{b\}}$. In other words, $K(H) \ot^\alpha K(H)$ is the unique
 minimal closed ideal which is contained in every non-zero closed
 ideal of $B(H) \ot^\alpha B(H)$.  Therefore, in all cases, $L$ must
 contain the product ideal $K(H) \ot^\alpha K(H)$.

\end{proof}

\section{Closed Lie ideals of $A \omin C(X)$} \label{a-cx}

 Let $X$ be a compact Hausdorff space and $A$ be a unital
 $C^*$-algebra. It is well known (\cite[$\S$ 5]{guichardet}) that the
 canonical map $A \omin C(X) \ni a \ot f \mapsto f(\cdot) a \in C (X,
 A)$ extends to a unital $C^*$-isomorphism.  We will be using this
 fact and its consequences in the following observations.

    We first recall, from \cite{Mar95}, certain naturally arising
    closed ideals and closed Lie ideals of $A \omin C(X)$. Some of the
    proofs were not given in \cite{Mar95}.  For the sake of
    completeness and convenience, we provide outlines of those proofs
    in bigger generality. The following folklore observation for a
    UHF $C^*$-algebra was used in \cite[Theorem 3.1]{Mar95}. We
    include the details for a more general situation.

\begin{prop}\label{ideals-CXA}
Let $X$ be a compact Hausdorff space and $A$ be a topologically simple 
$C^*$-algebra. Then every closed ideal in $C (X, A)$ is of the form
$\{ f \in C(X, A): f (s) = 0 \ \mathrm{for\ all}\ s \in F\}$ for some
closed subset $F$ of $X$.
\end{prop}

\begin{proof}
Let $I$ be a closed ideal in $C(X,A)$. From \Cref{simple-ideal} and
the well known fact that every closed ideal in $C(X)$ is of the form
$J(F) := \{ f \in C(X): f (s) = 0 \ \mathrm{for\ all}\ s \in F\}$ for
some closed subset $F$ of $X$, $I$ corresponds to the ideal $A \omin
J(F)$ in $A \omin C(X)$. It is enough to show that $I =
\widetilde{J}(F)$ where $\widetilde{J}(F) := \{ f \in C(X, A): f (s) =
0 \ \mathrm{for\ all}\ s \in F\}$, which is clearly a closed ideal in
$C(X, A)$.

Clearly $I \subseteq \widetilde{J}(F)$. To obtain the equality we just
need to show that $I$ is dense in $\widetilde{J}(F)$. Let $f \in
\widetilde{J}(F) $ and $\epsilon > 0$. For each $a \in A$, consider the
open ball $B_\epsilon(a) : = \{ x \in A : \| x - a \| < \epsilon\}$
and the punctured open ball $B_\epsilon^\times(a) : = \{ x \in A
\setminus \{0\} : \| x - a \| < \epsilon\}$. The collection $\{
f^{-1}(B_\epsilon^\times(f(x)) ) : x \in X \setminus F\}\cup
\{f^{-1}(B_\epsilon(0))\}$ is an open cover of $X$. Fix a finite
subcover, say, $ \{ f^{-1}(B_\epsilon^\times(f(x_i) )) : 1 \leq i \leq
n \}\cup \{f^{-1}(B_\epsilon(0))\}$. Since $X$ is compact and
Hausdorff, there exists a partition of unity $\{\varphi_i : 0 \leq i
\leq n\}$ such that $\mathrm{supp} (\varphi_i) \subseteq
U_i:=f^{-1}(B_\epsilon^\times(f(x_i) ))$ for $1 \leq i \leq n$ and
$\mathrm{supp} (\varphi_0) \subseteq U_0:=
f^{-1}(B_\epsilon(0))$. Then, $\varphi_i \in J(F)$ for all $1 \leq i
\leq n$, so that $\sum_{i=1}^n f(x_i) \ot \varphi_i \in A \ot
J(F)$. Fix an $x_0 \in F$. Then, for each $x \in X$, we have
\begin{eqnarray*}
\|f(x) - \sum_{i=1}^n \varphi_i(x) f(x_i) \| & = & \|f(x) \sum_{i=0}^n \varphi_i(x) -
\sum_{i=0}^n \varphi_i(x) f(x_i)\|\\
& \leq & \sum_{i=0}^n \|f(x) - f(x_i) \|  \varphi_i(x) \\
& = &   \sum_{i\,  :\,  x\, \in\, U_i} \|f(x) - f(x_i) \| \varphi_i(x) \\
& < &  \epsilon.
\end{eqnarray*}
In particular, $\|f - \sum_{i=1}^n  f(x_i) \varphi_i \| < \epsilon$,
implying that $I$ is dense in $\widetilde{J}(F)$.
\end{proof} 

For a Lie ideal $L$ in a unital $C^*$-algebra $A$, a subspace $S$ and
an ideal $J$ in $C(X)$, it is easily seen that $L \ot J + \C 1 \ot S$
is a Lie ideal in $A \ot C(X)$. Note that, if $L$, $J$ and $S$ are
closed, then it is not clear whether the sum $\overline{L \ot J} + \C
1 \ot S$ is closed or not. Marcoux, in \cite[Theorem 3.1]{Mar95}, had
shown that for a UHF $C^*$-algebra $A$, the sum $\overline{sl(A) \ot
  J} + \C 1 \otimes S$ is always closed.  Exploiting Marcoux's
technique, we prove the same in a more general setting. Before that,
we first make the following observation:

\begin{lemma}\label{sla-ot-J}
Let $X$ be a compact Hausdorff space and $A$ be a $C^*$-algebra with
$\mathcal{T}(A) \neq \emptyset$. Then, for any closed set $F$ in $X$,
the closed Lie ideal $L(F):=\ol{sl(A) \ot J(F)}$ in  $ A \omin C(X)$
corresponds to the closed Lie ideal $$\widetilde{L}(F) :=\{ f \in C(X, A):
f(s) = 0\ \mathrm{for\ all}\ s \in F\ \mathrm{and}\  \varphi \circ f =
0\ \mathrm{for\ all}\ \varphi \in \mathcal{T}(A)\}$$ in $C(X,A)$.
  \end{lemma}
\begin{proof}
Clearly, under the canonical $*$-isomorphism between $ A \omin C(X)$
and $C(X, A)$, the closed Lie ideal ${L}(F)$ is mapped onto a closed
Lie ideal in $\widetilde{L}(F)$. It just remains to show that the
image is dense in $\widetilde{L}(F)$. Let $f \in \widetilde{L}(F)$ and
$\epsilon > 0$. Then, $f \in \widetilde{J}(F)$, and as in the proof of
\Cref{ideals-CXA}, there exist finite sets $\{x_1, \ldots, x_n\}
\subseteq X \setminus F$ and $\{\varphi_1, \ldots, \varphi_n\}
\subseteq C(X)$ such that $\varphi_i(F) = \{0\}$ for all $1 \leq i
\leq n$ and $\| f - \sum_{i = 1}^n f(x_i) \varphi_i\| <
\epsilon$. Since $sl(A) = \cap_{\varphi \in \mcal{T}(A)} \ker
(\varphi)$ and $f \in \widetilde{L}(F)$, it readily follows that
$\sum_{i = 1}^n f(x_i) \ot \varphi_i \in sl(A) \ot J(F)$ and we are
done.
  \end{proof}

An adaptation of Marcoux's proof, gives us
the following generalization of the sufficient condition of
\cite[Theorem 3.1]{Mar95}.

\begin{prop}\label{lie-ideals}
  Let $X$ be a compact Hausdorff space and $A$ be a unital
  $C^*$-algebra.  If $\mcal{T}(A) \neq \emptyset$, then a subspace of
  the form $ L = \overline{sl(A) \ot J} + \C 1 \otimes S$, where $J$
  is a closed ideal and $S$ is a closed subspace in $C(X)$, is a
  closed Lie ideal in the $C^*$-algebra $A \omin C(X)$.
\end{prop}

\begin{proof} 
It is easy to verify that $L = \overline{sl(A) \ot J} + \C 1 \ot S $ is a
  Lie ideal. We only need to show that $L$ is closed.  Now, let $\{f_n\}$
  be a sequence in $L$ converging to some $f$ in $A \omin
  C(X)$. Decompose $f_n = g_n+ h_n$, where $g_n \in \ol{sl(A) \ot J}$ and
  $h_n \in \C 1 \ot S$. Since $A \omin C(X) $ is isometrically
  isomorphic to $C(X,A)$, we can assume that $\varphi \circ f_n,
  \varphi \circ f \in C(X)$ for all $\varphi \in \mcal{T}(A)$ and $n
  \geq 1$.

Clearly $\{\varphi \circ f_n\}$ is uniformly
convergent to $\varphi \circ f$ for all $\varphi \in
\mcal{T}(A)$. Since $g_n \in \ol{sl(A) \ot J}$, we have $\varphi \circ
f_n = \varphi \circ h_n$, so that $\varphi \circ h_n$ also converges
uniformly to $ \varphi \circ f$ for all $\varphi \in \mcal{T}(A)$.

Since $h_n \in \C 1 \ot S$, there exists a $\mu_n \in C(X)$ such
that $h_n(x) = \mu_n(x) 1$ for all $ x \in X$; so that $ \varphi \circ
h_n = \mu_n$ for all $\varphi \in \mcal{T}(A)$ and $ n \geq 1$. This
implies that $\{h_n\}$ is Cauchy and hence converges uniformly to some
$h$ in $A \omin C(X)$. Since $S$ is closed, $\C 1 \ot S$ is closed and
we have $h \in \C 1 \ot S$.

Finally, $\{g_n\} $ converges uniformly to $f-h =g$ (say) in $C(X,
A)$. Since $\varphi \circ g_n = 0$ for all $\varphi \in \mcal{T}(A)$
and $ n \geq 1$, we have $\varphi \circ g = 0$ for all $\varphi \in
\mcal{T}(A)$. Also, if $X_J := \{ x \in X: f (x) =
0\ \mathrm{for\ all}\ f \in J\}$, then $g_n(X_J) =\{0\}$ for all $n$
and hence $g(X_J) = \{0\}$ as well. Therefore, by \Cref{sla-ot-J}, $g \in
\overline{sl(A) \ot J}$ and $L$ is closed.
\end{proof}

In the reverse direction, as yet another application of the
$C^*$-isomorphism between $A \omin C(X)$ and $C(X, A)$, we will have
two instances to appeal to the following observation (from
\cite[Theorem 3.1]{Mar95}):

\begin{lemma}\label{NI}
 Let $X$ be a compact Hausdorff space and $A$ be a simple unital
 $C^*$-algebra.  For any closed ideal $I$ in $A \omin C(X)$, we have
 $$N(I) = I + \C 1 \otimes C(X).$$
\end{lemma}
\begin{proof}
By \Cref{simple-ideal}, $I$ is of the form $A \omin J$ for some closed
ideal $J$ in $C(X)$. Clearly, $A \omin J+ \C 1 \otimes C(X) \seq N(A
\omin J)$. And if $ f \in N(A \omin J) \seq C(X, A)$, then $[f, g] \in
A \omin J$ for all $g \in C(X, A)$. In particular, for each $a \in A$,
if $f_a \in C(X, A)$ denotes the constant function taking the value
$a$, then $f(x)a - a f(x) = 0$ for all $x \in F$, where $F$ is the
closed set in $X$ that determines the closed ideal $J$, i.e., $J(F) =
J$. Thus, $f(x) \in \mcal{Z}(A) = \C 1$ for all $x \in F$. If $g :=
f_{|_F}$, then after identifying $C(X)$ with $\C1 \ot C(X)$, by
Tietze's Extension Theorem, $g$ can be extended to a scalar-valued map
$\tilde{g}$ on $X$ so that $\tilde{g} \in \C 1 \ot C(X)$. Since $f (x)
- \tilde{g} (x) = 0$ for all $ x \in F$, we have $f - \tilde{g} \in A
\omin J$, so that $ f \in A \omin J + \C 1 \ot C(X).$
  \end{proof}
    
\begin{lemma}\label{AI-L-NI}
 Let $X$ be a compact Hausdorff space, $A$ be a simple unital
 $C^*$-algebra and $J$ be a closed
 ideal in $C(X)$. Then,
 \begin{enumerate}
 \item $ \ol{[A \omin C(X), A \omin J]} = \left\{ \begin{array}{ll}
   \overline{sl(A) \ot J} & \mathrm{if}\ \mcal{T}(A) \neq \emptyset,
   \ \mathrm{and}\\ A \omin J & \mathrm{if}\ \mcal{T}(A) =
   \emptyset.\end{array}\right.$
\item If $A$ admits a unique tracial state and $L$ is a closed subspace of $A \omin C(X)$ satisfying
 $$ \ol{[A \omin C(X), A \omin J]} \seq L \seq N(A \omin J),
  $$ then $L$ is a closed Lie ideal and is of the form $L = \ol{sl( A) \ot J}
  + \C 1 \otimes S$ for some closed subspace $S$ in $C(X)$.
\end{enumerate}
 \end{lemma}
 
\begin{proof}

$(1)$ By \Cref{commutator}, we have $ \ol{[A \omin C(X), A \omin J]} =  \ol{[A
     \omin J, A \omin J]}$ and it is easily verified that $ \ol{[A
     \omin J, A \omin J]} = \overline{\ol{[A,A]} \ot J}$.
 Therefore, by \Cref{cp}, we obtain the desired forms for $ \ol{[A \omin C(X), A \omin J]}$.
 \vspace*{2mm}

 $(2)$ By \Cref{commutator}, we also have $N(A \omin J) = N(\ol{[A \omin
     C(X), A \omin J]})$, and therefore $L$ is a closed Lie ideal.  The
 fact that  $L$ must be of above form follows on the lines of a part of the
 proof of \cite[Theorem 3.1]{Mar95}. We, therefore, just mention the
 steps involved and omit the details:

From $(1)$, we see that $\overline{sl(A) \ot J} \seq L$. By \Cref{NI},
we have $N(A \omin J) = A \omin J + \C 1 \otimes C(X)$, and since
$\mcal{T}(A)$ is a singleton, we also have $A = \C 1 \oplus sl(A)$.
Using \Cref{ideals-CXA} and \Cref{sla-ot-J}, one then deduces that, in
fact, $N(A\omin J) = \overline{sl(A) \ot J} + \C 1 \ot C(X)$. And,
therefore, since $L$ is a closed subspace satisfying
$$ \overline{sl(A) \ot J} \seq L \seq \overline{sl(A) \ot J} + \C 1
\ot C(X),$$ by \Cref{lie-ideals}, we
must have $L = \overline{sl(A) \ot J} + \C 1 \ot S$ for some closed
subspace $S$ of $C(X)$. 
\end{proof}

Note that a subspace of the form $ L = \overline{A \otimes J + \C 1
  \otimes S}$ where $J$ is a closed ideal and $S$ is a subspace of
$C(X)$ is clearly a closed Lie ideal of the $C^*$-algebra $A \omin
C(X)$.  The crux of the next theorem is that all closed Lie ideals of
$A \omin C(X)$ arise as in \Cref{lie-ideals}, the first part of which
is a generalization of the necessary condition of \cite[Theorem
  3.1]{Mar95}.

\begin{thm}
 Let $X$ be a compact Hausdorff space and $A$ be a simple unital
 $C^*$-algebra with at most one tracial state.  Then a subspace $L$ of
 $A \omin C(X)$ is a closed Lie ideal if and only if
 $$ L = \left\{
 \begin{array}{ll}
     \overline{sl(A) \ot J} + \C 1 \ot S & \mathit{if}\ \mcal{T}(A)
     \neq \emptyset,\ \mathit{and} \\ \overline{A \otimes J + \C 1 \ot
       S} & \mathit{if}\ \mcal{T}(A) = \emptyset,
 \end{array}
 \right.    $$
     for some  closed  ideal $J$ and  closed subspace $S$ in $C(X)$. 
\end{thm}

\begin{proof}
 We just need to prove the only if part in both cases. Let $L$ be a closed Lie
 ideal in $A \omin C(X)$. 

$(1)$ Suppose $\mcal{T}(A) \neq \emptyset$. Since $A \omin C(X)$ is
 a $C^*$-algebra and $A$ is simple, by \Cref{ideal-lie-ideal} and
 \Cref{simple-ideal}, there exists a closed ideal $I = A \omin J $ in
 $A \omin C(X)$ for some closed ideal $J$ in $C(X)$, such that
 $$ \ol{[A \omin C(X), A \omin J]} \seq L \seq N(A\omin J).$$ $L$ then
 has the required form by \Cref{AI-L-NI}.
  \vspace*{1mm}
  
$(2)$ Suppose $A$ has no  tracial states. Since $A \omin C(X)$ is
  a $C^*$-algebra, and $A$ embeds in $A \omin C(X)$ as a
  $C^*$-subalgebra, $A \omin C(X)$ also does not admit any tracial
  state. Therefore, since $A$ is simple, by \Cref{ideal-lie-ideal}, \Cref{commutator1} and
  \Cref{simple-ideal}, there exists a closed ideal $I = A \omin J $ of $A
  \omin C(X)$ for some closed ideal $J$ of $C(X)$, such that
$$ A \omin J \seq L \seq N(A\omin J).$$ Again, as in $(1)$, by
  \Cref{NI}, we see that $N(A \omin J) = A \omin J + \C 1 \otimes
  C(X)$. Now, $A \omin J$ being closed, so is $N(A\omin J)$ implying
  that ${A \omin J + \C 1 \otimes C(X)}$ is closed. In particular, $L$
  must be of the form $$ L = \overline{A \otimes J + \C 1 \ot S }$$
  for some closed subspace $S$ in $C(X)$. 
 \end{proof}

\section{Closed Lie ideals of $A \ot^\alpha K(H)$}\label{cput}

In this section, we will  analyze the Lie ideals of the tensor product spaces
$A \oh K(H)$ and $A \oop K(H)$. For this, we need an auxillary result from
Bre\v{s}ar et. al (\cite{bresar08}). For the sake of completeness, we
include a short discussion on the pre-requisites. Throughout this section, 
 $H$ will be assumed to be a seperable Hilbert space.

Given any unit vector $e$ in a Hilbert space $H$, one considers the
rank one orthogonal projection $p_e : H \ra H$ given by $ p_e (x)
=\langle x , e \rangle e, x \in H$. Then, for any finite orthonormal
system $\mcal{E} = \{e_i : 1 \leq i\leq n\}$ in a Hilbert space $H$,
consider the orthogonal projection $p_{\mcal{E}}:= \sum_i p_{e_i}$ and
the completely positive maps maps $s_{\mcal{E}}, t_{\mcal{E}} : K(H)
\ra K(H)$ given by $ s_{\mcal{E}} (x) = \sum_i p_{e_i} x p_{e_i}$ and
$t_{\mcal{E}} (x) = p_{\mcal{E}} x p_{\mcal{E}}$ for $x \in K(H)$.

Clearly $t_{\mcal{E}}$ is a complete contraction on $K(H)$ and the same is true
about $s_{\mathcal{E}}$,  which can  be seen  as follows.
\begin{lemma}\label{s-elon}
For every finite orthonormal system $\mcal{E} = \{e_i : 1 \leq i\leq
n\}$ in a Hilbert space $H$, $ s_{\mcal{E}}$ is a complete contraction on
$K(H)$.
\end{lemma}
\begin{proof}
It is enough to show that the map $B(H) \ni x
\stackrel{\varphi}{\mapsto} \sum_i p_{e_i} x p_{e_i} \in B(H)$ is a
complete contraction. Consider the orthogonal decomposition $H =
\oplus_{i = 1}^n p_{e_i}H$. Then, for each $x \in B(H)$, $\sum_i
p_{e_i} x p_{e_i}$ corresponds to the $n \times n$ diagonal matrix
operator $\mathrm{diag}(p_{e_1}xp_{e_1}, \cdots,$ $
p_{e_n}xp_{e_n})$. Therefore, $\|\sum_i p_{e_i} x p_{e_i}\| \leq \|
x\|$ for all $x \in B(H)$. Now, for $[x_{ij}] \in M_n(B(H))$, we
have $$\varphi^{(n)} ([x_{ij}]) = [ \varphi(x_{ij})] = \sum_i Q_i
[x_{ij}] Q_i$$ where $Q_i := \mathrm{diag}( p_{e_i}, p_{e_i}, \ldots,
p_{e_i})$ is a projection in $M_n(B(H)) = B(H^{(n)})$. So, as above
$\varphi^{(n)}$ is a contraction on $B(H^{(n)})=M_n(B(H))$ for all $n
\geq 1$.
  \end{proof}

   A tensor norm
$\alpha$  is said to be {\it completely positive uniform} if for every
completely positive map $T_i : A_i\ra B_i$, $i = 1, 2$, the canonical
linear map $T_1 \otimes T_2: A_1 \ot A_2 \ra B_1 \ot B_2$ has a
continuous extension $T_1\otimes^{\alpha} T_2: A_1 \otimes^{\alpha}
A_2 \ra B_1 \otimes^{\alpha} B_2$ satisfying $\|T_1 \otimes^{\alpha}
T_2\| \leq \|T_1\|\|T_2\|$.

Many known tensor norms are completely positive uniform including the
$C^*$-injective norm $\|\cdot\|_{\min}$, the $C^*$-projective norm
$\|\cdot\|_{\max}$, the Haagerup norm $\|\cdot\|_h$, the operator
space projective norm $\|\cdot\|_{\wedge}$ and the Banach space
injective and projective norms $\|\cdot\|_{\lambda}$ and
$\|\cdot\|_{\gamma}$ - see \cite{blecher, ERbook,  Tak}.

The following is an analogue of \cite[Corollary 5.22]{bresar08} and
follows directly from their deep result (\cite[Theorem
  5.15]{bresar08}) which involves some serious algebraic techniques.

\begin{prop}\label{main-thm}
   Let $A$ be a unital $C^*$-algebra and $\alpha$ be a completely
   positive uniform algebra tensor norm. Then a closed subspace of the
   Banach algebra $ A \otimes^{\alpha} K(H) $ is a Lie ideal if and
   only if it is a closed ideal and has the form $\overline{ I \ot K(H)}^{\alpha}$
   for some closed ideal $I$ in $A$.
 \end{prop}

 \begin{proof}
Let  $\mathcal{E}$ be  a finite orthonormal system in $H$. Since
 $\alpha$ is completely positive uniform, and $s_{\mathcal{E}},
 t_{\mathcal{E}}$ are completely positive and completely contractive,
 the maps $S_{\mathcal{E}}:= \mathrm{id}\ot
 s_{\mathcal{E}},\ T_{\mathcal{E}}:= \mathrm{id}\ot t_{\mathcal{E}} :
 A \ot F(H) \ra A \ot F(H)$ extend continuously on $A \ot^{\alpha}
 K(H)$ and satisfy $ \|S_{\mathcal{E}} \|, \|T_{\mathcal{E}} \| \leq
 1$. Therefore, by \cite[Theorem 5.15]{bresar08}, every closed Lie
 ideal in $ A \otimes^{\alpha} K(H) $ is a closed ideal and has the
 form $\overline{ I \ot K(H)}^{\alpha}$ for some closed ideal $I$ of
 $A$.
 \end{proof}

 \begin{cor}
If $A$ is a unital $C^*$-algebra and $\alpha$ is either the Haagerup
norm or the operator space projective norm, then any closed Lie ideal
of $A \ot^\alpha K(H) $ is precisely of the form $I \ot^\alpha K(H)$
for some closed ideal $I$ in $A$.
 \end{cor}
 \begin{proof}
The Haagerup norm $\|\cdot\|_h$ and the operator space projective norm
$ \|\cdot\|_{\wedge} $ are completely positive uniform algebra tensor
norms (\cite[Proposition 2]{blecher} and \cite[Theorem
  3.2]{ERbook}). The assertion made in the statement then follows from
\Cref{main-thm}, the injectivity of $\ot_h$ and the fact that $I \oop
K(H) \cong \ol{I \ot K(H)} \seq A \oop K(H)$ (\cite[Theorem
  5]{Kum01}).
 \end{proof}

 \noindent{\bf Acknowledgements.}
 The authors would like to thank Leonel Robert for bringing to our
 notice his recent work \cite{Rob} which was instrumental in the
 improvement of \Cref{q-c-a-i}.


\end{document}